\documentclass[12pt]{extarticle}
\usepackage{amsmath, amsthm, amssymb, hyperref, color}
\usepackage[shortlabels]{enumitem}
\usepackage{graphicx}
\usepackage[all]{xypic}
\usepackage{makecell}

\oddsidemargin \evensidemargin
\newtheorem{theorem}{Theorem}
\newtheorem{proposition}[theorem]{Proposition}
\newtheorem{lemma}[theorem]{Lemma}
\newtheorem{corollary}[theorem]{Corollary}

\newtheorem{problem}[theorem]{Problem}
\newtheorem{remark}[theorem]{Remark}
\newtheorem{example}[theorem]{Example}
\newtheorem{conjecture}[theorem]{Conjecture}

\newcommand{\RR}{\mathbb{R}}

\newcommand{\PP}{\mathbb{P}}

  \date{}

\title{\textbf{Algebraic Identifiability of Gaussian Mixtures}}
\author{Carlos Am\'endola, Kristian Ranestad and Bernd Sturmfels}

\begin{document}

\maketitle

\begin{abstract} \noindent
We prove that all moment varieties of univariate Gaussian mixtures have the expected dimension. 
Our approach rests on intersection theory and
Terracini's classification of defective surfaces.
The analogous identifiability result  is shown to be false for mixtures of Gaussians in dimension three and higher.
Their moments up to third order define projective varieties that are defective.
Our geometric study suggests an extension of
the Alexander-Hirschowitz Theorem for Veronese varieties
to the Gaussian setting.
\end{abstract}

\section{Introduction}

The Gaussian moment variety $\mathcal{G}_{n,d} $ is a subvariety of $\PP^N$, where $N = \binom{n+d}{d}-1$. 
Following  \cite{CB}, its points are the vectors of all
moments of order $\leq d$ of an  $n$-dimensional Gaussian distribution, parametrized birationally
 by the entries of the mean  vector $\mu = (\mu_1,\ldots,\mu_n)$
and  the covariance matrix $\,\Sigma = (\sigma_{ij})$.
The variety  $ \mathcal{G}_{n,d}$ is rational of dimension $n(n+3)/2$ for $d \geq 2$.    
 Its $k$th secant variety ${\rm Sec}_k (\mathcal{G}_{n,d})$
is the Zariski closure in $\PP^N$ of the set of vectors
of moments of order $\leq d$ of any probability distribution on $\RR^n$
that is the mixture of $k$ Gaussians, for $k \geq 2$.
Our aim is to determine the dimension 
of  the secant variety ${\rm Sec}_k (\mathcal{G}_{n,d})$.

That dimension is always bounded above by the number of parameters, so we have
\begin{equation} \label{eq:secdim}
 \dim \bigl( {\rm Sec}_k (\mathcal{G}_{n,d}) \bigr)
\,\,\leq \, \, \min \left\lbrace \,N\,,\,\, \, kn(n+3)/2  \,+\, k-1 \,\right\rbrace. 
\end{equation}
The right hand side is the \textit{expected dimension}. 
If equality holds in (\ref{eq:secdim}), then 
${\rm Sec}_k \big(\mathcal{G}_{n,d}$) is {\em nondefective}.
If this holds, and $N \geq \frac{1}{2}kn(n+3)  + k-1 $, then
the Gaussian mixtures are  algebraically identifiable
from their $N$ moments of order $\leq d$.
Here {\em algebraically identifiable}
means that the map from the model parameters to the moments is generically finite-to-one.
This means parameters can  be recovered by solving a 
zero-dimensional system of polynomial equations.
The term {\em rationally identifiable} is used if the map is generically one-to-one.

We focus our attention on algebraic identifiability.
In this paper we do not  study rational identifiability.
We prove the following result
that contrasts the cases $n=1$ and $n \geq 3$.

\begin{theorem} \label{thm:main}
Equality holds in (\ref{eq:secdim}) for $n=1$ and all values of $d$ and $k$.
Hence all moment varieties of mixtures of univariate Gaussians are algebraically
identifiable.
The same is false for  $n \geq 3$, $d=3$ and $k=2$: here the right hand side of (\ref{eq:secdim})
exceeds the left hand side by two.
\end{theorem}

 Defective  Veronese varieties  are classified by
 the celebrated {\em Alexander-Hirschowitz Theorem} \cite{AH}.
 This is relevant for our discussion because each Veronese variety is 
 naturally contained in a corresponding Gaussian moment variety.
 The latter is a noisy version of the former, since the Veronese variety consists
 of the points on  $\mathcal{G}_{n,d}$ where the covariance matrix is zero.
  We refer to \cite[Section 6]{CB}. Remark \ref{rmk:shaowei}
  discusses other fixed covariance matrices.
   Note that Theorem~\ref{thm:main} proves the first part of Conjecture 15 in \cite{CB} about algebraic identifiability, and it also disproves the generalized ``natural conjecture'' stated after Problem~17 in~\cite{CB}.

Our result for $d=3$ is a Gaussian analogue of the infinite family
($d=2$) in the Alexander-Hirschowitz classification \cite{AH} of defective Veronese varieties.
Many further defective cases for $d=4$ are exhibited in 
Table \ref{tab:defective4} and Conjecture \ref{conj:eleven}.
Extensive computer experiments (up to $d=24$) suggest that moment  varieties are never defective for 
bivariate Gaussians ($n=2$).

\begin{conjecture} \label{conj:two}
Equality holds in (\ref{eq:secdim}) for $n=2$ and all values of $d$ and $k$.
In particular, all moment varieties of mixtures of bivariate Gaussians are algebraically identifiable.
\end{conjecture}

Our presentation is organized as follows. In Section 2 we focus on the case $n=1$. 
We review basics on the Gaussian moment surfaces $\mathcal{G}_{1,d}$,
and what is known classically on defectivity of surfaces. Based on this,
we then prove the first part of Theorem \ref{thm:main}.
In Section~3 we study our problem for $n \geq 2$. We begin with 
the parametric representation of ${\rm Sec}_k(\mathcal{G}_{n,d})$,
we next establish the second part of Theorem~\ref{thm:main},
and thereafter we study the defect and we examine higher moments.
Section 4 discusses what little we know   about the degree and equations of the varieties 
${\rm Sec}_k(\mathcal{G}_{n,d})$. Both Sections 3 and 4 feature many open problems.

\section{One-dimensional Gaussians}

The moments $m_0,m_1,m_2,\ldots,m_d$ of a Gaussian distribution on the real line
are polynomial expressions in the mean $\mu$ and the variance $\sigma^2$.
These expressions will be reviewed in Remark~\ref{rmk:familiar}.
They give a parametric representation
of the Gaussian moment surface $\mathcal{G}_{1,d}$ in $\PP^d$.
The following implicit representation of that surface was derived in
\cite[Proposition 2]{CB}.

\begin{proposition}
\label{prop:surface}
Let $d \geq 3$.
The homogeneous prime ideal of the Gaussian moment surface $\mathcal{G}_{1,d}$ 
is minimally generated by $\binom{d}{3}$ cubics.
These are the $3 \times 3$-minors of the $3 \times d$-matrix
$$ 
G_d \,\,= \, \left(\begin{array}{ccccccc} 
    0&m_0&2m_1&3m_2 & 4m_3 & \cdots & (d-1) m_{d-2}\\ 
    m_0& m_1& m_2 & m_3 & m_4 & \cdots & m_{d-1}\\
    m_1& m_2& m_3 & m_4 & m_5 &\cdots & m_d\\
    \end{array}\right).
$$  
The $3 \times 3$-minors of the matrix $G_d$ form a Gr\"obner basis
for the prime ideal of  $\mathcal{G}_{1,d} $ with respect to the 
reverse lexicographic term order. 
This implies that $\,\mathcal{G}_{1,d}$ has degree $\binom{d}{2}$ in $\PP^d$.
\end{proposition}

Our first new result concerns the singular locus on the Gaussian moment surface.

\begin{lemma} \label{lem:Gsing}
The singular locus of the surface $\,\mathcal{G}_{1,d}\,$ is the line defined by
$\langle m_0,m_1,\ldots,m_{d-2} \rangle$.
\end{lemma}

\begin{proof}
Let $\mathcal{L}$ be the line defined by
$\langle m_0,m_1,\ldots,m_{d-2} \rangle$ and 
$\mathcal{S} = {\rm Sing}(\mathcal{G}_{1,d} )$.
We claim $\mathcal{L} = \mathcal{S}$.

We first show that $\mathcal{S} \subseteq \mathcal{L}$.
Consider the affine open chart $\{m_0 = 1\}$ of $\,\mathcal{G}_{1,d}$.
On that chart,  the coordinates $m_i$
are polynomial functions in the unknowns $m_0,\ldots,m_{i-1}$, for $i \geq 3$.
Indeed,  the $3 \times 3$-minor of $G_d$
with column indices $1,2$ and $i$ has the form
$m_i - h(m_0,\ldots,m_{i-1})$. Hence
$\mathcal{G}_{1,d} \,\cap \,\{m_0 = 1\}  \simeq \mathbb{A}^2$,
and therefore $\mathcal{S} \subset \{m_0 = 0\}$.
Next suppose $m_0 = 0$. The leftmost
$3 \times 3$-minor of $G_d$ implies $m_1 = 0$. Now,
the minor with columns $2,3,4$ implies that $m_2 = 0$,
the minor with columns $3,4,5$ implies that $m_3 = 0$, etc.
From the rightmost minor we conclude $m_{d-2} = 0$.
This shows that  $\,\mathcal{G}_{1,d} \,\cap \,\{m_0 = 0\} = \mathcal{L}$,
and we conclude $\mathcal{S} \subseteq \mathcal{L}$.

For the reverse inclusion $\mathcal{L} \subseteq \mathcal{S}$,
we consider the Jacobian matrix of the cubics that define
$\mathcal{G}_{1,d}$. That matrix has $d+1$ rows
and $\binom{d}{3}$ columns. We claim that it has rank
$\leq d-3$ on~$\mathcal{L} $. To see this, 
note that the term $m_i m_{d-1}^2$ appears in the 
minor of $G_d$ with columns $i,d-1,d$ for $i=2,\ldots,d-2$,
and that all other occurrences of $m_{d-1}$ or $m_d$
in any of the $3 \times 3$-minors of $G_d$ is linear. Therefore the 
Jacobian matrix restricted to $\mathcal{L}$
has only $d-3$ non-zero entries, and so its rank
is at most $d-3$. This is less than
$d-2 = {\rm codim}(\mathcal{G}_{1,d})$.
We conclude that all points on the line $\mathcal{L}$
are singular points in the Gaussian moment surface $\mathcal{G}_{1,d}$.
\end{proof}

The $3 \times d$-matrix $G_d$ has entries that are linear forms in
$d+1$ unknowns $m_0,\ldots,m_{d}$. That matrix may be interpreted 
as a $3$-dimensional tensor of format $3 \times d \times (d+1)$. That tensor can
be turned into a $d \times (d+1)$ matrix whose entries are linear
forms in three unknowns $x,y,z$. The result is 
what we call the \textit{Hilbert-Burch matrix} of
our surface $\mathcal{G}_{1,d}$. 
It equals
\begin{equation}
\label{eq:hilbertburch}
B_d \,\,= \, \left(\begin{array}{cccccc} 
    y&z&0&0 & \cdots & 0\\ 
    x& y& z & 0& \cdots & 0\\
    0& 2x& y & z & \cdots & 0\\
    0& 0 & 3x &y & \cdots & 0 \\
    \vdots & \vdots & \vdots & \vdots & & \vdots\\
    0 & 0 & \cdots &(d-1)x & y & z\\
    \end{array}\right).
\end{equation}
Its maximal minors generate a
Cohen-Macaulay ideal, defining a
scheme $Z_d$ of length $\binom{d+1}{2}$ supported at the point $(1:0:0)$. 
Consider the map defined by the maximal minors of $B_d$,
$$\phi: \PP^2 \dashrightarrow \PP^d .$$ 
The base locus of the map $\phi$
is the scheme $Z_d$  and its image is the surface $\mathcal{G}_{1,d}$.  
\begin{remark} \label{rmk:familiar}
\rm
The parametrization $\phi$ onto $\mathcal{G}_{1,d}$ is birational. It equals the familiar
affine parametrization, as in (\ref{eq:gaussian}), of the Gaussian moments in terms of mean and variance if we set
\begin{equation}
\label{eq:xyz}
 x = -\sigma^2 \, , \,\,\,
y = \mu 
\,\,\, {\rm and} \,\,\,
z = 1. 
\end{equation}
The image 
of the line $\{x=-\sigma^2z\}$, for fixed value of the variance $\sigma^2$, is a 
rational normal curve of degree $d$
 inside the Gaussian moment surface $\mathcal{G}_{1,d}$.
It is defined by the $2\times 2$-minors of a $2$-dimensional space of rows in the matrix $G_d$.  
The singular line $\mathcal{L}\subset \mathcal{G}_{1,d}$ is the tangent line to this curve at the point $(0:\cdots:0:1)$. In particular, the image of the line $\{x=0\}$ is the rational normal curve defined by the $2\times 2$-minors of the last two rows of $G_d$.
\end{remark}

We now come to our main question, namely whether there exist $d$ and $k$ such that
$\mathcal{G}_{1,d}$ is $k$-defective in $\PP^d$.  Theorem \ref{thm:main} asserts
that this is not the case. Equivalently,  the dimension~of
${\rm Sec}_k (\mathcal{G}_{1,d})$ is always
equal to the minimum of $d$ and $3k-1$,
which is the upper bound  in~(\ref{eq:secdim}).

Curves can never be defective, but surfaces can.
The prototypical example 
is the Veronese surface $S $ in the space  $\PP^5$ of
symmetric $3 \times 3$-matrices. Points on $S$ are matrices of rank $1$.
The secant variety ${\rm Sec}_2(S)$ consists of matrices of rank $\leq 2$. Its expected dimension is five
whereas the true dimension of $S$ is only four. This means that $S$ is $k$-defective for $k=2$.

The following well-known result on higher secant varieties of a variety $X$ allows us to show that $X$ is not 
$k$-defective for any $k$ by proving this for one particular $k$ (see \cite{Adl}):

\begin{proposition}\label{defect}
Let $X$ be a $k'$-defective subvariety of $\PP^d$ and $k>k'$.  Then $X$ is $k$-defective as long as 
${\rm Sec}_{k}(X)$ is a proper subvariety of $ \PP^d$. In fact, the defectivity increases with~$k$:
\begin{equation} \label{eq:leftminusright}
({\rm dim} (X)+1) \cdot k-1-{\rm dim}({\rm Sec}_k(X)) \,\,>\,\, ({\rm dim} (X)+1) \cdot k'-1-{\rm dim}({\rm Sec}_{k'}(X)).
\end{equation}
\end{proposition}

\begin{proof} By Terracini's Lemma, the dimension of the secant variety ${\rm Sec}_{k}(X)$ is the dimension 
of the span of the tangent spaces to $X$ at $k$ general points.  Since $X$ is $k'$-defective and $k'<k$,
 the linear span of $k-k'$ general tangent spaces to  the affine cone over $X$ must intersect the span of $k'$
 such general tangent spaces in a positive-dimensional linear space.  The dimension of that intersection is 
 the difference of the left hand side minus the right hand side in (\ref{eq:leftminusright}).
 \end{proof}

\begin{corollary}\label{maxk} If a surface $X\subset \PP^d$ is  defective, 
then $X$ is $k$-defective for some $k\geq (d-2)/3$.
\end{corollary}

\begin{proof}
We proceed by induction on $k$. If the surface $X$ is $(k-1)$-defective and $k< (d-2)/3$, 
then ${\rm dim} ({\rm Sec}_{k}(X))<3k+2 < d$.  So $X$ is also $k$-defective, by Proposition \ref{defect}.
\end{proof}

Our main geometric tool is Terracini's 1921 classification  of all $k$-defective surfaces:

\begin{theorem}(Classification of $k$-defective surfaces) \label{thm:terra}
Let $X \subset \PP^N$ be a reduced, irreducible, non-degenerate projective surface that is $k$-defective. Then $k \geq 2$ and either
\begin{enumerate}
\item[(1)] $X$ is the quadratic Veronese embedding of a rational normal surface $Y$ in $\PP^k$; or
\item[(2)] $X$ is contained in a
  cone over a curve, with apex a linear space of dimension $\leq k-2$.
\end{enumerate}
Furthermore, for general points $x_1,\ldots,x_k$ on $X$
there is a hyperplane section tangent along a curve $C$ that passes through these points.
In case (1), the curve $C$ is irreducible; 
in case (2), the curve $C$ decomposes into $k$ algebraically equivalent curves 
$C_1,\ldots,C_k$ with $x_i\in C_i$.
\end{theorem}
\begin{proof}  See  \cite[Theorem 1.3 (i),(ii)]{CC} and cases (i) and (ii) of the proof given there.
\end{proof}

Chiantini and Ciliberto  offer a nice historical account
of this theorem in the introduction to their article \cite{CC}.
A modern proof follows from the more general result in \cite[Theorem 1.1]{CC}.

 \begin{corollary} \label{cor:statement2}
If the surface $ X = \mathcal{G}_{1,d}$ is $k$-defective, then
statement (2) in Theorem \ref{thm:terra} holds. 
 \end{corollary}

\begin{proof}
We need to rule out case (1) in Theorem \ref{thm:terra}.
A rational normal surface is either a Hirzebruch surface or it is
the cone over a rational curve. The former is smooth
and the latter is singular at only one point. 
The same is true for the quadratic Veronese embedding of such a surface.
By contrast, our surface $\mathcal{G}_{1,d}$ is
singular along a line, by Lemma \ref{lem:Gsing}.
Alternatively, a quadratic Veronese embedding of a surface contains no line.
\end{proof}

Our goal is now to rule out case (2) in Theorem \ref{thm:terra}.
That proof will be much more involved.
 Our strategy is to set up a system of surfaces and morphisms
 between them, like this:
 \begin{equation}
 \label{eq:Sd}
\begin{array}{ccccc}
 S_d & \to  & \bar {S_d}&\subset&\PP^{N_d }\\
 \downarrow &   &\downarrow &&   \\
 \PP^2 &   & \mathcal{G}_{1,d}&\subset &\PP^d   
\end{array}
\end{equation}
The second row in (\ref{eq:Sd}) represents the rational map
$\phi : \PP^2  \dashrightarrow \mathcal{G}_{1,d}$
that is given by the maximal minors of $B_d$. 
Above $\PP^2$ sits a smooth surface $S_d$ which we shall
construct by a sequence of blow-ups from $\PP^2$. It will
have the property that $\phi$ lifts to a morphism on $S_d$.
Curves of degree $d$ in $\PP^2$ specify a divisor class $H_d$ on $S_d$.
The complete linear system $|H_d|$ maps $S_d$ onto
a rational surface  $\bar{S_d}$
in $\PP^{N_d}$ where $N_d = {\rm dim}(|H_d|)$.
The subsystem of $|H_d|$ given by the $d+1$ maximal 
minors of $B_d$, then defines the vertical map from
$\bar{S_d}$ onto $\mathcal{G}_{1,d}$. 
Our plan is to use the intersection theory on $S_d$ to 
rule out the possibility (2) in Theorem \ref{thm:terra}.

\begin{lemma}\label{decomp}
Suppose that we have a diagram as in (\ref{eq:Sd}) 
and $X = \mathcal{G}_{1,d}$ satisfies statement (2) in Theorem \ref{thm:terra}.
Then, for any $k$ general points $x_1,\ldots,x_k$ on the surface $S_d$, there exist 
linearly equivalent divisors $D_1 \ni x_1 ,\ldots,D_k \ni x_k$
and there exists a hyperplane section of 
$\,\mathcal{G}_{1,d} $ in $\PP^d$, with pullback $H_d$ to $S_d$, such that
$H_d - 2 D_1  - 2D_2- \cdots - 2D_k $ is effective on $S_d$.
\end{lemma}

 
\begin{proof}  By part (2) of Theorem \ref{thm:terra}, there exist algebraically equivalent 
curves $C_1,\ldots, C_k$ on $X$ 
that contain the images of the respective points $x_1,\ldots,x_k$, and there is
a hyperplane section $H_X$ of $X$ which contains and is singular along 
each $C_i$. Let $H\subset S_d$ be the preimage of $H_X$, and let $D_i\subset S_d$ be the preimage of $C_i$.  
Then $x_i \in D_i$ for $i=1,\ldots,k$. Furthermore,
the divisor $H$ has multiplicity at least $2$ along each $D_i$.  
Finally, since $S_d$ is a rational surface, linear and algebraic equivalence of divisors coincide,
and the lemma follows.
\end{proof}

We now construct the smooth surface $S_d$. Let $V_d$
denote the $(d+1)$-dimensional vector space spanned by
the maximal minors of the matrix  $B_d$ in (\ref{eq:hilbertburch}).
When $d$ is odd these minors~are
$$ \begin{matrix}
b_{d,0}&= & z^d,\\
b_{d,1}&= & yz^{d-1},\\
b_{d,2}&= & y^2z^{d-2}-xz^{d-1},\\
b_{d,3}&= & y^3z^{d-3}-3xyz^{d-2}, \smallskip \\
\cdots & \cdots & \cdots \quad \cdots \quad \cdots \quad \cdots \smallskip \\
b_{d,d-1}&= & y^{d-1}z-\binom{d-1}{2}xy^{d-3}z^2+ \ldots
+a_{(\frac{d-3}{2},d-1)}x^{\frac{d-3}{2}}y^2z^{\frac{d-1}{2}}+a_{(\frac{d-1}{2},d-1)}x^{\frac{d-1}{2}}z^{\frac{d+1}{2}}, \smallskip \\
b_{d,d}&= & y^d-\binom{d}{2}xy^{d-2}z+a_{(2,d)}x^2y^{d-4}z^2 +
\,\ldots \,+a_{(\frac{d-1}{2},d)}x^{\frac{d-1}{2}}yz^{\frac{d-1}{2}}.
\end{matrix}
$$
When $d$ is even, the maximal minors of the Hilbert-Burch matrix $B_d$  are
$$ \begin{matrix}
b_{d,0}&= & z^d,\\
b_{d,1}&= & yz^{d-1},\\
b_{d,2}&= & y^2z^{d-2}-xz^{d-1},\\
b_{d,3}&= & y^3z^{d-3}-3xyz^{d-2}, \smallskip \\
\cdots & \cdots & \cdots \quad \cdots \quad \cdots \quad \cdots \smallskip \\
b_{d,d-1}&= & y^{d-1}z-\binom{d-1}{2}xy^{d-3}z^2+ \ldots +a_{(\frac{d-4}{2},d-1)}x^{\frac{d-4}{2}}y^3z^{\frac{d-2}{2}}+a_{(\frac{d-2}{2},d-1)}x^{\frac{d-2}{2}}yz^{\frac{d}{2}}, \smallskip \\
b_{d,d}&= & y^d-\binom{d}{2}xy^{d-2}z+a_{(2,d)}x^2y^{d-4}z^2 + \,\ldots\,
 + a_{(\frac{d}{2},d)}x^{\frac{d}{2}}z^{\frac{d}{2}}.
\end{matrix}
$$
Here the $a_{(i,j)}$ are rational constants. The point $p=(1:0:0)$ is the
only  common zero of the forms $b_{d,0},\ldots, b_{d,d}$. All forms are
 singular at $p$, with the following lowest degree terms:
 \begin{equation}
 \label{eq:analysis1}
 \qquad z^d,\,yz^{d-1}, \,z^{d-1}, \, yz^{d-2}, \, \ldots \,,\,  z^{(d+1)/2},\,yz^{(d-1)/2}
\qquad \qquad \hbox{when $d$ is odd}; 
\end{equation}
\begin{equation}
\label{eq:analysis2}
  \qquad z^d,\,yz^{d-1}, \,z^{d-1}, \,yz^{d-2}, \, \ldots\, ,\,  yz^{d/2},\,z^{d/2}
 \qquad \qquad \hbox{when $d$ is even}. 
 \end{equation}
 Consider a general form in $V_d$. Then its lowest degree term at $p$ is a 
 linear combination of $z^{(d+1)/2}$ and $yz^{(d-1)/2}$ when $d$ is odd, 
 and it is a scalar multiple of $z^{d/2}$ when $d$ is even.
 
 The forms $b_{d,0},\ldots,b_{d,d}$
 define a morphism  $\phi: \PP^2 \backslash \{p\}\to \PP^d$ that does not extend to $p$.  
 Consider any map $\pi:S'\to \PP^2$ that is obtained by
 a sequence of blow-ups at smooth points, starting with the blow-up of $\PP^2$ at $p$.  
 Let $E \subset S'$ be the preimage of $p$.  The restriction of $\pi$ to  $S' \backslash E$ 
  is an isomorphism onto $\PP^2 \backslash \{p\}$, 
  and so $\phi$ naturally defines a morphism $S' \backslash E\to \PP^d$.  
  
 We now define our surface $S_d$ in (\ref{eq:Sd}).
It is a minimal surface $S'$  such that 
 $S' \backslash E\to \PP^d$ extends to a morphism $\tilde{\phi}:S'\to \PP^d$.
 Here ``minimal'' refers to the number of blow-ups, and we do not claim $S_d$ is the unique such minimal surface.

Let $H_d$ be the strict transform on $S_d$ of a curve in $\PP^2$ defined by a general form in $V_d$.  The complete linear system $|H_d|$ on $S_d$ defines a morphism $S_d\to \PP^{N_d}$, 
where $N_d = {\rm dim} |H_d|$.  Let $\bar{S_d}\subset \PP^N$ be the image.  Then $\tilde{\phi}:S_d\to \PP^d$ is the composition of $S_d\to \PP^N$ and a linear projection to $\PP^d$ whose restriction to $\bar{S_d}$ is finite. Thus we now have the diagram in (\ref{eq:Sd}). 

Relevant for proving Theorem~\ref{thm:main} are
the first two among the blow-ups that lead to $S_d$.
  The map $\phi$ is not defined at $p$.  More precisely, $\phi$ is undefined 
  at $p$ and at its tangent direction $\{z=0\}$.
 Let $S_p\to \PP^2$ be the blow-up at $p$, with exceptional divisor $E_p$.
  Let $S_{p,z}\to S_{p}$ be the  blow-up at the point on $E_p$ corresponding to the
   tangent direction $\{z=0\}$ at $p$,  with exceptional divisor $E_z$. 
 To obtain $S_d$ we need to blow up $S_{p,z}$ in $s$  further points for some $s$.
 
 Now, $S_d$ is a smooth rational surface.  Let $L$ be the class of a line pulled back to 
 $S_d$, and let $E_p,E_z,F_1,\ldots,F_s$,  be the classes of the exceptional divisors of 
 each blow-up,  pulled back to $S_d$. The divisor class group of $S_d$ 
 is the free abelian group with basis $L,E_p,E_z,F_1,\ldots ,F_s$.  
 The intersection pairing on this group is diagonal for this basis, with
  \begin{equation}
 \label{eq:pairing}
 L^2\,= \,-E_p^2 \,= \, -E_z^2 \, = \, -F_1^2 \, = \, \cdots \,= \, -F_s^2 \,\, = \,\, 1.
 \end{equation}
 The intersection of two curves on the smooth surface $S_d$, having 
 no common components, is a nonnegative integer.
 It is computed as the intersection pairing of their classes using (\ref{eq:pairing}).

 \begin{lemma} \label{lem:eleven}
 Consider the linear system $|H_d|$ on $S_d$ that represents 
 hyperplane sections of $ \mathcal{G}_{1,d} \subset \PP^d$,
 pulled back via the morphism $\tilde \phi$. Its class in the
 Picard group of $S_d$ is given by
 $$ \begin{matrix}
 H_d & =  & dL- \frac{d}{2} E_p- \frac{d}{2}E_z\,\,-\, c_1 F_1 -c_2 F_2 - \,\cdots\, - c_s F_s 
  & \hbox{when $d$ is even,} \smallskip \\
H_d & = & dL- \frac{d+1}{2}E_p- \frac{d-1}{2}E_z
\,-\, c_1 F_1 -c_2 F_2 - \,\cdots\, - c_s F_s &  \hbox{when $d$ is odd.}
\end{matrix}
$$
Here $c_1,c_2,\ldots,c_s$ are positive integers whose precise value will not matter to us.
\end{lemma}

\begin{proof} 
The forms in $V_d$ define the preimages in $\PP^2$ of curves in $|H_d|$.
The first three coefficients are seen from the analysis 
in (\ref{eq:analysis1}) and (\ref{eq:analysis2}). 
  The general hyperplane in $\PP^d$ intersects the image of the exceptional curve $F_i$ in 
  finitely many points. Their number is the coefficient $c_i$.
    \end{proof}

\begin{proof}[Proof of the first part of Theorem~\ref{thm:main}]
Suppose that $X = \mathcal{G}_{1,d}$ is $k$-defective for some $k$.
By Corollary \ref{maxk}, we may assume that $3k+2\geq d$.
By Corollary \ref{cor:statement2}~and Lemma \ref{decomp},
the class of the linear system $|H_d|$ in the Picard group of 
the smooth surface $S_d$ can be written as
$$H_d\,\,=\,\,A\,+\,2kD , $$
where $A$ is effective and $D$ is the class of a curve on $S_d$ that  has 
no fixed component. According to Lemma \ref{lem:eleven}, we can write
$$ D\,=\,aL-b_pE_p-b_zE_z-\sum_{i=1}^sc_i'F_i, $$
where $a=D\cdot L$ is a positive integer and 
$b_p, b_z,c'_1,\ldots,c'_s$ are nonnegative integers.

Assume first that  $a\geq 2$. We have the following chain of inequalities:
$$ 0 \,\leq \, L\cdot A \,=\, L \cdot H_d -  2k (L \cdot D)
\,=\, d - 2ka \,\leq\, d-4k \,\leq \, 2-k. $$
This implies $k \leq 2$. The case  $k=1$ being vacuous, we conclude
that $k = 2$ and hence $d\leq 8$.
If $d \leq 5$, then ${\rm Sec}_2(\mathcal{G}_{1,d}) = \PP^d$ is easily checked,
by computing the rank of the Jacobian of the parametrization.
For $d=6$, we know from \cite[Theorem 1]{CB} that
 ${\rm Sec}_2(\mathcal{G}_{1,6}) $ is a hypersurface of degree $39$ in $\PP^6$.
If $d \in \{7,8\}$, then the secant variety  ${\rm Sec}_2(\mathcal{G}_{1,d}) $ is
also $5$-dimensional, by the computation with cumulants in \cite[Proposition 13]{CB}.

Next, suppose $a=D\cdot L=1$.
The divisor $D$ is the strict transform on $S_d$ of a line in $\PP^2$.  
The multiplicity of this line at $p$ is at most $1$, i.e. $0\leq D\cdot E_p\leq 1$.  Furthermore,
 $D\cdot E_z=0$ because $D$ moves.
Suppose that $D\cdot E_p=0$ and $d$ is even. Then we have 
 $d\geq 4k$ because
$$ d/2\,= \, H_d \cdot E_p \,= A\cdot E_p\,\leq \,A\cdot L\,\leq \, d-2k. $$
Since $d\leq 3k+2$, this implies $k=2$ and $d=8$.
This case has already been ruled out above.
If $D\cdot E_p=0$ and $d$ is odd, then 
the same reasoning yields $(d+1)/2=A\cdot E_p \leq d-2k$.
 This implies $3k+2\geq d\geq 4k+1$, which is impossible for $k \geq 2$.

It remains to examine the case $D\cdot E_p=1$.
Here, any curve linearly equivalent to  $D$ on $S_d$  is the strict transform of a line in $\PP^2$ passing through $p=(1:0:0)$.  
Through a general point in the plane there is a unique such line, so it suffices to show that the doubling of any line through $p$ is not a component of any curve defined by  a linear combination of the $b_{d,i}$.  In particular, it suffices to show that  $y^2$ is not a factor of any form in the vector space $V_d$.

To see this, we note that no monomial $x^r y^s z^t$ appears in
more than one of the forms $b_{d,0}, b_{d,1}, \ldots,b_{d,d}$.
Hence, in order for $y^2$ to divide a linear combination of 
$b_{d,0}, b_{d,1}, \ldots,b_{d,d}$, it must already  divide one of the $b_{d,i}$. 
However, from the explicit expansions we see that $y^2$ is not a factor of
  $b_{d,i}$ for any $i$. This completes the proof of the first part 
    in Theorem~\ref{thm:main}.
\end{proof}

\section{Higher-dimensional Gaussians}

We begin with the general definition of the moment variety for Gaussian mixtures.
The coordinates on $\PP^N$ are the moments $m_{i_1 i_2 \cdots i_n}$.
The variety ${\rm Sec}_k (\mathcal{G}_{n,d})$ has the parametrization
\begin{equation}
\label{eq:gaussian}
\sum_{i_1,i_2,\ldots,i_n \geq 0}
\frac{m_{i_1 i_2 \cdots i_n}}{i_1 ! i_2 ! \cdots i_n !}
t_1^{i_1} t_2^{i_2} \cdots t_n^{i_n} \,\, =\,\,
\sum_{\ell = 1}^k \lambda_\ell \,\cdot\,
{\rm exp}(t_1 \mu_{\ell 1} + \cdots + t_n \mu_{\ell n}) \cdot
{\rm exp} \biggl( \frac{1}{2} \sum_{i,j=1}^n \sigma_{\ell ij} t_i t_j \! \biggr).
\end{equation}
This is a formal identity of generating functions in $n$ unknowns $t_1,\ldots,t_n$.
The model parameters are the $kn$ coordinates $\mu_{\ell i}$ of the mean vectors,
the $k \binom{n+1}{2}$ entries $\sigma_{\ell i j}$ of the covariance matrices,
and the $k$ mixture parameters $\lambda_\ell$. The latter 
satisfy $\lambda_1 + \cdots + \lambda_k = 1$.
This is a map from the space of model parameters into the
affine space $\mathbb{A}^{\! N}$ that sits inside $\PP^N$ as
$\{m_{00\cdots 0} = 1\}$.
We define ${\rm Sec}_k (\mathcal{G}_{n,d}) \subset \PP^N$ as the projective closure
of the image of this map.

\begin{remark}\rm  The affine Gaussian moment variety $\mathcal{G}_{n,d}\cap \mathbb{A}^{\! N}$ is isomorphic to an affine space (cf. \cite[Remark 6]{CB}).
In particular it is smooth.
Hence the singularities of $\mathcal{G}_{n,d}$ are all contained in the hyperplane at infinity.  This means that the definition of ${\rm Sec}_k (\mathcal{G}_{n,d})$ 
is equivalent to the usual definition of higher secant varieties: 
it is  the closure of the union of all
$(k-1)$-dimensional linear spaces that intersect $\mathcal{G}_{n,d}$ in $k$ distinct smooth points.
\end{remark}

In this section we focus on the case $d=3$, that is, we examine the 
varieties defined by first, second and third moments of Gaussian distributions.
The following is our main result.

\begin{theorem} \label{thm:mainsec2}
The moment variety $\mathcal{G}_{n,3}$ is $k$-defective for $ k \geq 2$.
 In particular, for $k=2$, the model
has two more parameters  than the dimension of the secant
variety, i.e.~$\,n(n + 3) + 1 \,-\, {\rm dim}\bigl({\rm Sec}_2(\mathcal{G}_{n,3})\bigr) = 2$.
If $n \geq 3$ and we fix  distinct first coordinates $\mu_{11}$ and $\mu_{21}$
for the two mean vectors, then the remaining parameters are identified uniquely.
In each of these statements, the parameter $k$ is assumed to be in the range where
${\rm Sec}_k(\mathcal{G}_{n,3})$  does not fill $\PP^N$.
\end{theorem}

This proves the second part of Theorem~\ref{thm:main}.
We begin by studying the first interesting~case.

\begin{example} \label{s2G33} \rm
Let $n=d=3$ and $k=2$. In words, we 
consider moments up to order three for the mixture
of two Gaussians in $\RR^3$. This case is special because the number of
parameters coincides with the dimension of the ambient space:
$\,N= \frac{1}{2} k n(n+3) + k - 1 = 19$. The variety ${\rm Sec}_2 (\mathcal{G}_{3,3})$ 
is the closure of the image of the map $\mathbb{A}^{\! 19} \rightarrow \PP^{19}$
that is given by (\ref{eq:gaussian}):
$$
\begin{matrix}
m_{100} & = &  \lambda \mu_{11} +(1-\lambda) \mu_{21} \\
m_{010} & = &  \lambda \mu_{12} +(1-\lambda) \mu_{22} \\
m_{001} & = &  \lambda \mu_{13} +(1-\lambda) \mu_{23}  \\
m_{200} & = &  \lambda (\mu_{11}^2+\sigma_{111})+(1-\lambda) (\mu_{21}^2+\sigma_{211}) \\
m_{020} & = &  \lambda (\mu_{12}^2+\sigma_{122})+(1-\lambda) (\mu_{22}^2+\sigma_{222})  \\
m_{002} & = &  \lambda (\mu_{13}^2+\sigma_{133})+(1-\lambda) (\mu_{23}^2+\sigma_{233}) \\
m_{110} & = &  \lambda (\mu_{11} \mu_{12}+\sigma_{112})
               +(1-\lambda) (\mu_{21} \mu_{22}+\sigma_{212}) \\
m_{101} & = &  \lambda (\mu_{11} \mu_{13}+\sigma_{113})
               +(1-\lambda) (\mu_{21} \mu_{23}+\sigma_{213}) \\
m_{011} & = &  \lambda (\mu_{12} \mu_{13}+\sigma_{123})
               +(1-\lambda) (\mu_{22} \mu_{23}+\sigma_{223}) \\
m_{300} & = &  \lambda (\mu_{11}^3+3 \sigma_{111} \mu_{11})
               +(1-\lambda) (\mu_{21}^3+3 \sigma_{211} \mu_{21}) \\
m_{030} & = &  \lambda (\mu_{12}^3+3 \sigma_{122} \mu_{12})
               +(1-\lambda) (\mu_{22}^3+3 \sigma_{222} \mu_{22}) \\
m_{003} & = &  \lambda (\mu_{13}^3+3 \sigma_{133} \mu_{13})
               +(1-\lambda) (\mu_{23}^3+3 \sigma_{233} \mu_{23}) \\
m_{210} & = &  \lambda (\mu_{11}^2 \mu_{12}+\sigma_{111} \mu_{12}+2 \sigma_{112} \mu_{11})
          +(1-\lambda) (\mu_{21}^2 \mu_{22}+\sigma_{211} \mu_{22}+2 \sigma_{212} \mu_{21}) \\
m_{201} & = &  \lambda (\mu_{11}^2 \mu_{13}+\sigma_{111} \mu_{13}+2 \sigma_{113} \mu_{11})
          +(1-\lambda) (\mu_{21}^2 \mu_{23}+\sigma_{211} \mu_{23}+2 \sigma_{213} \mu_{21}) \\
m_{120} & = &  \lambda (\mu_{11} \mu_{12}^2+\sigma_{122} \mu_{11}+2 \sigma_{112} \mu_{12})
          +(1-\lambda) (\mu_{21} \mu_{22}^2+\sigma_{222} \mu_{21}+2 \sigma_{212} \mu_{22}) \\
 m_{102} & = &  \lambda (\mu_{11} \mu_{13}^2+\sigma_{133} \mu_{11}+2 \sigma_{113} \mu_{13})
          +(1-\lambda) (\mu_{21} \mu_{23}^2+\sigma_{233} \mu_{21}+2 \sigma_{213} \mu_{23}) \\
m_{021} & = &  \lambda (\mu_{12}^2 \mu_{13}+\sigma_{122} \mu_{13}+2 \sigma_{123} \mu_{12})
          +(1-\lambda) (\mu_{22}^2 \mu_{23}+\sigma_{222} \mu_{23}+2 \sigma_{223} \mu_{22}) \\
m_{012} & = &  \lambda (\mu_{12} \mu_{13}^2+\sigma_{133} \mu_{12}+2 \sigma_{123} \mu_{13})
          +(1-\lambda) (\mu_{22} \mu_{23}^2+\sigma_{233} \mu_{22}+2 \sigma_{223} \mu_{23}) \\
m_{111} & = &  \lambda\, (\mu_{11} \mu_{12} \mu_{13}
              + \sigma_{112} \mu_{13} + \sigma_{113} \mu_{12} + \sigma_{123} \mu_{11})
                \qquad \qquad\        \\ & & 
          +\,(1 - \lambda) (\mu_{21} \mu_{22} \mu_{23}
                        +\sigma_{212} \mu_{23}+\sigma_{213} \mu_{22}+\sigma_{223} \mu_{21}) \qquad \qquad
                        \qquad \quad\,\,
\end{matrix}
$$
A direct computation shows that the $19 \times 19$-Jacobian matrix
of this map has rank $17$ for generic parameter values.
Hence the dimension of ${\rm Sec}_2(\mathcal{G}_{3,3})$ 
equals $17$. This is two less than the expected dimension of $19$.
We have here identified the smallest instance of defectivity.

Let $m = (m_{ijk}) $ be a valid vector of moments.
Thus $m$ is a point in ${\rm Sec}_2(\mathcal{G}_{3,3})$.
We assume that $m \not\in \mathcal{G}_{3,3}$.
Choose arbitrary but distinct complex numbers for $\mu_{11} $ and $\mu_{21}$,
while the other $17$ model parameters remain unknowns.
We note that, if $\mu_{11} = \mu_{21}$, then 
$m_{300}=3m_{100}m_{200}-2m_{100}^3$.
This is not satisfied for a general choice of 19 model parameters.

What we see above is a system of $19$ 
polynomial equations in $17$ unknowns. We claim 
that this system has a unique solution over $\mathbb{C}$.
Hence, if $\mu_{11},\mu_{21} \in \mathbb{Q}$ and 
  the left hand side vector $m$ has its coordinates in $\mathbb{Q}$, then that
  unique solution has its coordinates in $\mathbb{Q}$.

By solving the first equation, we obtain the mixture parameter $\lambda$.
From the second and third equation we can eliminate $\mu_{12}$ and $\mu_{13}$.
Next, we observe that all $12$ covariances $\sigma_{ijk}$ appear linearly
in our equations, so we can solve for these as well. 
We are left with a system of truly non-linear equations in only two unknowns,
$\mu_{22}$ and $\mu_{23}$. A direct computation now reveals that this
system has a unique solution that is a rational expression in the given $m_{ijk}$.

Our computational argument therefore shows that each general fiber of the
natural parametrization of ${\rm Sec}_2(\mathcal{G}_{3,3})$
is birational to the affine plane $\mathbb{A}^2$ 
whose coordinates are $\mu_{11}$ and $\mu_{21}$.
This establishes  Theorem  \ref{thm:mainsec2} for the special case 
of trivariate Gaussians ($n=3$).
\end{example}

\begin{remark}\label{g23} \rm
The second assertion in Theorem  \ref{thm:mainsec2} holds
for $n = 2$ because there are $11$ parameters and ${\rm Sec}_2(\mathcal{G}_{2,3}) = \PP^9$.
However, the third assertion is not true for $n=2$ because
the general fiber of the  parametrization map $\mathbb{A}^{11} \rightarrow \PP^9$
is the union of three irreducible components. When $\mu_{11}$ and $\mu_{21}$
are fixed, then the fiber consists of three points and not one.
\end{remark}

\begin{proof}[Proof of Theorem \ref{thm:mainsec2}]
Suppose  $n \geq 4$ and let $m \in  {\rm Sec}_2(\mathcal{G}_{n,3})
\backslash \mathcal{G}_{n,3}$. Each moment $m_{i_1 i_2 \cdots i_n}$
has at most three non-zero indices. Hence, its expression in the model
parameters involves at most three coordinates of the mean vectors
and a block of size at most three in the covariance matrices.
Let $\mu_{11}$ and $\mu_{21}$ be arbitrary distinct complex numbers.
Then we can apply the rational solution in Example~\ref{s2G33} for any $3$-element subset
of $\{1,2,\ldots,n\}$ that contains $1$. This leads to unique expressions
for all model parameters in terms of the moments $m_{i_1 i_2 \cdots i_n}$.
In this manner,  at most one system of parameters is recovered.
Hence the third sentence in Theorem \ref{thm:mainsec2} is implied by the first two sentences.
It is these two we shall now prove.

In the affine space $\mathbb{A}^{\! N}= \{m_{000}=1\} \subset \mathbb{P}^N$,
we consider the affine moment variety $G_n^A:=\mathcal{G}_{n,3}\cap 
\mathbb{A}^{\! N}$. This has dimension $M=\frac{1}{2}n(n+3)$.
  The  map from (\ref{eq:gaussian}) that parametrizes the Gaussian moments is denoted
  $ \rho: \mathbb{A}^{\! M}\to \mathbb{A}^{\! N}$.
It is an isomorphism onto its image $G_n^A$.  

 Fix two points $p=(\mu,\sigma)$ and $p'=(\mu',\sigma')$ in $ \mathbb{A}^{\! M}$.
They determine  the affine plane
$$A(p,p')\,\,=\,\,\bigl\{\,(s\mu+(1-s)\mu',\,t\sigma+(1-t)\sigma')\,\,|\,\, s,t\in \RR\, \bigr\} \,\,\subset\,\,
\mathbb{A}^{\! M}.$$
Its image $\rho(A(p,p'))$ is a surface in $G_n^A \subset \mathbb{A}^{\! N}$.   
The restrictions $m_{i_1 \ldots i_n}(s,t)$  of the moments  to this surface are
 polynomials in $s,t$ with coefficients that depend on the points $p,p'$.  
 Since $i_1+ \cdots +i_n \leq 3$, every moment
$m_{i_1 \ldots i_n}(s,t)$ is a linear combination of the monomials $1,s,t,st,s^2,s^3$.
Linearly eliminating these monomials, we
obtain $N{-}5$ linear relations among the moments when restricted to the plane $A(p,p')$.  These relations define the affine span
of the surface $\rho(A(p,p'))$. This affine space is therefore
  $5$-dimensional. We denote it by $\mathbb{A}^5_{p,p'}$. 
  
  The monomials $(b_1,b_2,b_3,b_4,b_5) = (s,t,st,s^2,s^3)$ serve as coordinates
  on  $\mathbb{A}^5_{p,p'}$, modulo the affine-linear relations that define  $\mathbb{A}^5_{p,p'}$, 
       The image surface $\rho(A(p,p'))$ is therefore contained in the subvariety 
  of $\mathbb{A}^5_{p,p'}$ that is defined by the $2\times 2$-minors of the $ 2 \times 4$-matrix
  \begin{equation}
  \label{eq:twobyfour1}
  \begin{pmatrix}
  1 & b_2 & b_1  & b_4 \\   b_1 & b_3  & b_4 & b_5
  \end{pmatrix} \,\, = \,\, 
  \begin{pmatrix}
  1 & t & s & s^2 \\  s & st  & s^2 & s^3
  \end{pmatrix}.
  \end{equation}
      This variety is an irreducible surface,
  namely a  scroll of degree $4$. It hence equals $\rho(A(p,p'))$.
  
Let $\bar{\sigma}$ denote the covariance matrix with entries
$\, \bar{\sigma}_{ij}=(\mu_i-\mu_i')(\mu_j-\mu_j')$.
  We define
$$\mathbb{A}^3_{p,p'}  \,\,=\,\,\bigl\{\,(\mu'+s(\mu - \mu'),\,\sigma' +t(\sigma - \sigma') + u \bar{\sigma}) \,\,|\,\, s,t,u\in \RR\, \bigr\}. $$
Setting $u=0$ shows that this $3$-space contains the plane $A(p,p')$. We claim that
\begin{equation}
\label{eq:keyclaim}
\rho(\mathbb{A}^3_{p,p'}) \,\,\subseteq \,\,\mathbb{A}^5_{p,p'}.
\end{equation}
On the image $\rho(\mathbb{A}^3_{p,p'})$, each moment
is a linear combination of the eight monomials $1,s,s^2,s^3,t,st,u,su$. 
A key observation is that,
 by our choice of $\bar{\sigma}$, these expressions are actually linear combinations of the six
 expressions $1,s,s^2 {+}u,s^3{+}3su,t,st$.
 Indeed, the coefficient of $s^2$ in the expansion of
$(\mu'_i+s(\mu_i - \mu'_i))(\mu'_j+s(\mu_j - \mu'_j))$
matches the coefficient $ \bar{\sigma}_{ij}$ of $u$ in the expansion of second order moments. 
Likewise,  $s^2$ and $u$ have equal coefficients in the third order moments.
Analogously, the coefficient  of the monomial $s^3$ in the expansion~of 
$$(\mu'_i+s(\mu_i - \mu'_i))(\mu'_j+s(\mu_j - \mu'_j))(\mu'_k+s(\mu_k - \mu'_k))$$ 
is $(\mu_i-\mu'_i) \bar{\sigma}_{jk} = (\mu_j-\mu'_j) \bar{\sigma}_{ik} = (\mu_k-\mu'_k) \bar{\sigma}_{ij},$
which coincides with the corresponding coefficient of $3su$ in the expansion of third order moments.
From this we conclude that (\ref{eq:keyclaim}) holds.

Since $\rho$ is birational, $\rho(\mathbb{A}^3_{p,p'})$ is a threefold in $\mathbb{A}^5_{p,p'}$. 
Since $p$ and $p'$ are arbitrary, these threefolds cover $G^A_n$.
Through any point outside  $\rho(\mathbb{A}^3_{p,p'})$
 there is a $2$-dimensional family of secant lines to
$\rho(\mathbb{A}^3_{p,p'})$. The same holds for $G^A_n$. Hence
the $2$-defectivity of $\mathcal{G}_{n,3}$  is at least~two.

To see that it is at most two, it suffices to find a point $q$ in
${\rm Sec}_2(\mathcal{G}_{n,3})$ such that the variety of secant lines
to $\mathcal{G}_{n,3}$ through $q$ is $2$-dimensional.
Let $\mathcal{G}_{2,3}(1,2)$ denote the subvariety
of $\mathcal{G}_{n,3}$ defined  by setting all parameters other than
$\mu_1,\mu_2,\sigma_{11}, \sigma_{12},\sigma_{22}$ to zero.
The span of $\mathcal{G}_{2,3}(1,2) \cap \mathbb{A}^N$ is an affine
 $9$-space $\mathbb{A}^9(1,2)$ inside $\mathbb{A}^N$.
Consider a general point $q\in \mathbb{A}^9(1,2)$. Then $ q \not\in G^A_n$.
We claim that any secant to $G^A_n$ through $q$ is contained in
$\mathbb{A}^9(1,2)$. 

A computation with {\tt Macaulay2} \cite {MAC2}  shows
that this is the case when $n = 3$. Explicitly, if $q$ is any point whose
moment coordinates vanish except those that involve only
$\mu_1,\mu_2,\sigma_{11}, \sigma_{12},\sigma_{22}$,
then $\mu_3=\sigma_{13}=\sigma_{23}=\sigma_{33}=0$.
Suppose now $n \geq 4$. Assume there exists a secant line through $q$
that is not contained in $\mathbb{A}^9(1,2)$. Then we can find indices
$1,2,k$ such that the projection of that secant
 passes through the span of the corresponding
 $G^A_3\subset G^A_n$.
 In each case, the secant lands in $\mathbb{A}^9(1,2)$,
so it must already lie in this subspace before any of
the projections. This argument proves the claim.

 In conclusion,
  we have shown that the $2$-defectivity of the third order
 Gaussian moment variety $\mathcal{G}_{n,3}$  is precisely two.
 This completes the proof of Theorem~\ref{thm:mainsec2}.
 \end{proof}
 
 We offer some remarks on the geometry underlying the proof of Theorem~\ref{thm:mainsec2}, or more precisely, 
on the $2$-dimensional family of secant lines 
 through a general point $q$ on the affine secant variety  ${\rm Sec}_2(G_n^A)$.
 The {\em entry locus} $\Sigma_q$ is
the closure  of the set of points $p\in G_n^A$ such that 
$q$ lies on a secant line through $p$.    This entry locus is therefore a surface.  We identify the Zariski closure of this surface in $\PP^N$.
\begin{proposition}
The Zariski closure in $\mathcal{G}_{n,3}$ of the entry locus $\Sigma_q$ of a general point $q\in{\rm Sec}_2(G_n^A)$ is 
 the projection of a Del Pezzo surface of degree $6$ into $\PP^5$ that is singular along a line in the hyperplane at infinity.
\end{proposition}
 \begin{proof} 
 According to Example \ref {s2G33}, 
 the $2$-dimensional family of secant lines through a general point 
 $q\in{\rm Sec}_2(G_3^A)$ 
 is irreducible and birational to the affine plane. If we consider $G^A_3$ as a subvariety of $G^A_n$ and $q\in{\rm Sec}_2(G_3^A)$, then we may argue as in the proof of Theorem \ref{thm:mainsec2} that any secant line to $G^A_n$ through $q$, is a secant line to $G^A_3$.  We conclude that the $2$-dimensional family of secant lines through a general point 
 $q\in{\rm Sec}_2(G_n^A)$ is irreducible.
 

On the other hand, if $q$ is on the secant spanned by
$p,p' \in G_n^A$, then, in the notation of the proof of  Theorem \ref{thm:mainsec2}, the point $q$ lies in $\mathbb{A}^5_{p,p'}$.
There is a $2$-dimensional family of secant lines to $\rho(\mathbb{A}^3_{p,p'})$ through $q$.
This family must coincide with the family of secant lines to $G_n^A$ through $q$.
The entry locus $\Sigma_q$ therefore equals the double point locus of the projection 
$$\pi_q: \rho(\mathbb{A}^3_{p,p'})\to \mathbb{A}^4 $$ from the point $q$.  
We shall identify this double point locus as a surface of degree $6$. In fact, 
its Zariski closure in $\PP^5$ is the projection of a Del Pezzo surface of degree $6$ from $\PP^6$.

 Consider the maps
$$
\begin{matrix}
\tau\,:\,\mathbb{A}^3_{p,p'}\to \mathbb{A}^6 & : & \quad (s,t,u) & \mapsto & (s,t,st,s^2,s^3+3su, u) , \\
\pi\,:\,\mathbb{A}^6\to \mathbb{A}^5_{p,p'} & : &\quad (a_1,\ldots,a_6) & \mapsto& (a_1,a_2,a_3,a_4+a_6,a_5). 
\end{matrix}
$$
The image $\tau(\mathbb{A}^3_{p,p'})$ in $\mathbb{A}^6$ is the $3$-fold scroll defined by the $2\times 2$ minors of the matrix
\begin{equation}
  \label{eq:twobyfour2}
  \begin{pmatrix}1&a_2&a_1&a_4+3a_6\\
  a_1&a_3&a_4&a_5
  \end{pmatrix}.
  \end{equation}
  
The composition $\pi \circ \tau $ is the restriction of $\rho$ to $\mathbb{A}^3_{p,p'}$.
Hence $\rho(\mathbb{A}^3_{p,p'})$ is also a {\em quartic threefold scroll}.
To find its equations  in  $\mathbb{A}^5_{p,p'}$, we 
set $a_4 = b_4-a_6$ and $a_i = b_i$ for $i \in \{1,2,3,5\}$, and then we
 eliminate $a_6$ from the ideal of $2 \times 2$-minors of
 (\ref{eq:twobyfour2}). The result is the system
$$ b_1b_2-b_3 \,= \, 2b_1b_3^2+b_2^2b_5-3b_2b_3b_4 \,=\, 
2b_1^2b_3+b_2b_5-3b_3b_4 \, = \, 2b_1^3-3b_1b_4      +b_5 \,\,\, = \,\,\, 0. $$

Let  $X_{p,p'}$ be the Zariski  closure of $\tau(\mathbb{A}^3_{p,p'})$ in $ \PP^6$.  It is a threefold quartic scroll, defined by the
  $2\times 2$ minors of the matrix
\begin{equation}
  \label{eq:twobyfour2}
  \begin{pmatrix}a_0&a_2&a_1&a_4+3a_6\\
  a_1&a_3&a_4&a_5
  \end{pmatrix}.
  \end{equation}
The projection $\pi$, and the composition of $\pi$ and the projection $\pi_q$ from the point $q\in \mathbb{A}^5_{p,p'}$, extend to  projections  
$$\bar{\pi}:X_{p,p'}\to \PP^5\qquad \hbox{and} \qquad \tilde{\pi}:X_{p,p'}\to \PP^4.$$
By the double point formula \cite[Theorem 9.3]{Ful}, the double point locus $\Sigma_{\tilde{\pi}}\subset X_{p,p'}$ of $\tilde{\pi}$ is a surface of degree $6$ anticanonically embedded in $\PP^6$.  This is the desired Del Pezzo surface.  

Similarly, the double point locus of $\bar{\pi}$ is a plane conic curve in $X_{p,p'}$, that is mapped $2\colon 1$ onto a line in $\PP^5$.  The plane conic curve is certainly contained in the double point locus $\Sigma_{\tilde{\pi}}$, so  $\bar{\pi}(\Sigma_{\tilde{\pi}})\subset \PP^5$ is singular along a line.
In the above coordinates, the conic is the intersection of $X_{p,p'}$ with the plane defined by $a_0=a_1=a_2=a_3=0 $, i.e. a conic in the hyperplane $\{a_0=0\}$ at infinity.  
The entry locus $\Sigma_q$ is clearly contained in $\bar{\pi}(\Sigma_{\tilde{\pi}})$.  In fact, the latter is the Zariski closure of the former in $\PP^5$ and the proposition follows.
\end{proof}


We now come to the higher secant varieties of the Gaussian moment variety $\mathcal{G}_{n,3}$.

\begin{corollary} \label{cor:threedefective}
Let $k \geq 2$ and  $n \geq 3k-3$. Then $\mathcal{G}_{n,3}$ is $k$-defective.
\end{corollary}

\begin{proof}
This is immediate from Theorem \ref{thm:mainsec2} and Proposition \ref{defect}.
\end{proof}

\begin{table}[h]
\begin{center} \begin{tabular}{ | l | l | l | l | l | l | l | l | p{1.5cm} |} \hline  $n$ & $k$ & $d$ & par & $N$ & exp & $\dim$ & $\delta$ & par-$\dim$\\ 
\hline 5 & 3 & 3 & 62 & 55 & 55& 51 & 4 & 11 \\
\hline 6 & 3 & 3 & 83 & 83 & 83& 71 & 12 & 12 \\
\hline 6 & 4 & 3 & 111 & 83 & 83& 82 & 1 & 29 \\
\hline 7 & 3 & 3 & 107 & 119 & 107& 94 & 13 & 13 \\
\hline 7 & 4 & 3 & 143 & 119 & 119& 111 & 8 & 32 \\
\hline 8 & 3 & 3 & 134 & 164 & 134& 120 & 14 & 14 \\
\hline 8 & 4 & 3 &179 & 164 & 164& 144 & 20 & 35 \\
\hline 8 & 5 & 3 & 224 & 164 & 164 & 160 & 4 & 64 \\
\hline 9 & 3 & 3 & 164 & 219 & 164 & 149 & 15 & 15 \\
\hline 9 & 4 & 3 & 219 & 219 & 219 & 181 & 38 & 38\\
\hline 9 & 5 & 3 & 274 & 219 & 219 & 204 & 15 & 70 \\
\hline 10 & 3 & 3 & 197 & 285 & 197 & 181 & 16 & 16 \\
\hline 10 & 4 & 3 & 263 & 285 & 263 & 222 & 41 & 41 \\
\hline 10 & 5 & 3 & 329 & 285 & 285 & 253 & 32 & 76 \\
\hline 10 & 6 & 3 & 395 & 285 & 285 & 275 & 10 & 120 \\
\hline \end{tabular} 
\vspace{-0.11in}
\end{center}
  \caption{\label{tab:defective3}  Moment varieties 
  of order $d=3$ for mixtures of $k \geq 3$ Gaussians} \medskip
  \end{table}
  
Based on computations, like those in
Table~\ref{tab:defective3},
we propose the following conjecture.

\begin{conjecture} \label{conj:threedefective}
For any $n\geq 2$ and $k \geq 1$, we have
\begin{equation}
\label{eq:dimformula} \dim( {\rm Sec}_k(\mathcal{G}_{n,3}) ) 
\,\,=\,\, \frac{1}{6}k \left[ k^2 - 3(n+4)k + 3n(n+6) + 23 \right] - (n+2),
\end{equation}
for $k=1,2,\ldots,K$, where $K+1$ is the smallest integer such that the right hand side 
in {\rm (\ref{eq:dimformula})} is larger than the ambient dimension $\binom{n+3}{3}-1$.
\end{conjecture}

For $k=1$ this formula evaluates to ${\rm dim}(\mathcal{G}_{n,3}) = 
n(n+3)/2$, as desired. Conjecture \ref{conj:threedefective}
also holds for $k=2$. This is best seen by
rewriting the identity (\ref{eq:dimformula}) as follows:
$$  \frac{1}{2}kn(n+3)+k-1\, -\, {\rm dim}\bigl(
{\rm Sec}_k (\mathcal{G}_{n,3})\bigr) \,\, \, = \,\,\,
\frac{1}{2}(k-1)(k-2)n \,-\, \frac{1}{6}(k-1)(k^2-11k+6). $$
This is the difference between the expected dimension and the
true dimension of the $k$th secant variety. For $k=2$ this  equals $2$,
independently of $n$, in accordance with Theorem~\ref{thm:mainsec2}. 

Conjecture \ref{conj:threedefective} was verified computationally
for $n \leq 15$. Table \ref{tab:defective3} illustrates  all cases for
$n \leq 10$. Here,
${\rm exp} = {\rm min}({\rm par},N)$ is the {\em expected dimension}, and
$\delta = {\rm exp}-{\rm dim}$ is the {\em defect}.

We also undertook a comprehensive experimental study for
higher moments of multivariate Gaussians.
The  following two examples are the two smallest  defective cases for $d=4$.

\begin{example} \rm
Let $n=8$ and $d=4$. The Gaussian moment variety $\mathcal{G}_{8,4}$ is $11$-defective.
The expected dimension of ${\rm Sec}_{11} (\mathcal{G}_{8,4})$
equals the ambient dimension $N = 494$, but this secant
variety is actually a hypersurface in $\PP^{494}$. It would be very nice to know its degree.
\end{example}

\begin{example} \rm
Let $n=9$ and $d=4$. The moment variety $\mathcal{G}_{9,4}$ is $12$-defective
but it is not $11$-defective. Thus the situation is much more
complicated than that in Theorem \ref{thm:mainsec2},
where defectivity always starts at $k=2$. We do not yet have any theoretical
explanation for this.
\end{example}

Table \ref{tab:defective4} shows the first few defective cases 
for Gaussian moments of order $d=4$.
It suggests a clear pattern, resulting in the following conjecture.
We verified this for $n \leq 14$.

\begin{table}[h]
\begin{center} \begin{tabular}{ | l | l | l | l | l | l | l | l | p{1.5cm} |} \hline  $n$ & $k$ & $d$ & par & $N$ & exp & $\dim$ & $\delta$ & par-$\dim$\\ 
\hline 8 & 11 & 4 & 494 & 494 & 494 & 493 & 1 & 1 \\
\hline 9 & 12 & 4 & 659 & 714 & 659 & 658 & 1 & 1 \\
\hline 9 & 13 & 4 & 714 & 714 & 714 & 711 & 3 & 3 \\
\hline 10 & 13 & 4 & 857 & 1000 & 857 & 856 & 1 & 1 \\
\hline 10 & 14 & 4 & 923 & 1000 & 923 & 920 & 3 & 3 \\
\hline 10 & 15 & 4 & 989 & 1000 & 989 & 983 & 6 & 6 \\
\hline 11 & 14 & 4 & 1091 & 1364 & 1091 & 1090 & 1 & 1 \\
\hline 11 & 15 & 4 & 1169 & 1364 & 1169 & 1166 & 3 & 3 \\
\hline 11 & 16 & 4 & 1247 & 1364 & 1247 & 1241 & 6 & 6 \\
\hline 11 & 17 & 4 & 1325 & 1364 & 1325 & 1315 & 10 & 10 \\
\hline 12 & 15 & 4 & 1364 & 1819 & 1364 & 1363 & 1 & 1 \\
\hline 12 & 16 & 4 & 1455 & 1819 & 1455 & 1452 & 3 & 3 \\
\hline 12 & 17 & 4 & 1546 & 1819 & 1546 & 1540 & 6 & 6 \\
\hline 12 & 18 & 4 & 1637 & 1819 & 1637 & 1627 & 10 & 10 \\
\hline 12 & 19 & 4 & 1728 & 1819 & 1728 & 1713 & 15 & 15 \\
\hline 12 & 20 & 4 & 1819 & 1819 & 1819 & 1798 & 21 & 21 \\
\hline \end{tabular} 
\vspace{-0.11in}
\end{center}
  \caption{\label{tab:defective4} A census of defective Gaussian moment varieties $d=4$}
\end{table}

\begin{conjecture} \label{conj:eleven}
The Gaussian moment variety 
$\,\mathcal{G}_{n,4}\,$ is $(n+3)$-defective with defect $\delta_{n+3}=1$ for $n \geq 8$. 
Furthermore, for all $r \geq 3$, the $(n+r)$-defect  of $\,\mathcal{G}_{n,4}\,$
is equal to $\delta_{n+r}=\binom{r-1}{2}$, unless
the number of model parameters exceeds the ambient dimension $\binom{n+4}{4}-1$. \end{conjecture}

\section{Towards Equations and Degrees}

We begin Section 4 by reminding
 the reader that the Veronese variety
$\mathcal{V}_{n,d}$ is a subvariety of $\mathcal{G}_{n,d}$. It
 is obtained by setting the covariance matrix in the parametrization equal to zero.
The Gaussian moment variety can be thought of as a noisy version of the
Veronese variety. Indeed, points on $\mathcal{V}_{n,d}$ 
represent moments of order $\leq d$ of Dirac measures,
and points on its secant variety ${\rm Sec}_k(\mathcal{V}_{n,d})$ 
represent moments of  finitely supported signed measures on $\mathbb{R}^n$.

The celebrated Alexander-Hirschowitz Theorem \cite{AH} characterizes
defective Veronese varieties. It identifies all triples
$(n,d,k)$ such that a mixture of $k$ Dirac measures on $\mathbb{R}^n$ is not
algebraically identifiable from its moments of order $\leq d$. This section is a first step
towards a similar characterization for mixtures of $k$ Gaussian measures on $\mathbb{R}^n$.
The cases $d=3$ and $d=4$ for Gaussians, featured in Theorem~\ref{thm:mainsec2} and
Conjecture \ref{conj:eleven}, are reminiscent of the infinite family in the case $d=2$ for Veronese varieties.
At present we do not know any isolated defective examples that would be
analogous to the exceptional cases in the Alexander-Hirschowitz Theorem.

We wish to reiterate that the Gaussian moment varieties $\mathcal{G}_{n,d}$ are much more complicated than
the Veronese varieties $\mathcal{V}_{n,d}$. Beyond Proposition  \ref{prop:surface},
 their ideals are essentially unknown.
 
 In Remark \ref{rmk:familiar} we observed Veronese curves as subvarieties with fixed variance inside the Gaussian moment variety.  
 We record the following analogue for  higher-dimensional cases.
 
\begin{remark} \label{rmk:shaowei} \rm
Statisticians are often interested in Gaussian mixtures
where the entries $\sigma_{\ell i j}$ of the $k$ covariance matrices 
are fixed and the free parameters are the coordinates $\mu_{\ell i}$ of the
$k$ mean vectors. Such a model is a subvariety of  ${\rm Sec}_k(\mathcal{G}_{n,d})$
that is a {\em join of Veronese varieties}. Indeed, we see 
from  (\ref{eq:gaussian})  that any Gaussian moment variety with fixed covariance matrix
is isomorphic to the standard Veronese $\mathcal{V}_{n,d}$
after a linear change of coordinates in $\PP^N$,
and taking mixtures corresponds to taking joins.
In particular, if all $k$ covariance matrices are fixed and identical, then
the resulting moment variety is isomorphic to ${\rm Sec}_k(\mathcal{V}_{n,d})$
under a linear change of coordinates in $\PP^N$. Hence
the Alexander-Hirschowitz Theorem characterizes the
algebraic identifiability of Gaussian mixtures with fixed identical covariance matrix.
The varieties in this paper are new to geometers because the
covariance matrices are parameters.
\end{remark}

\smallskip

A well-known result in statistics states that, under reasonable hypotheses, probability 
distributions are determined by their moments. 
In addition, it is known (e.g.~from \cite{Yak}) that Gaussian mixtures are identifiable (in the statistical sense).
 Since their moments are polynomials in their parameters, Belkin and Sinha \cite{BS} concluded
  that (for $k$ and $n$ fixed) a finite set of moments is enough to recover the mixture model uniquely.
  In particular, the secant variety ${\rm Sec}_k( \mathcal{G}_{n,d})$ has
the expected dimension for $d \gg 0$ when $k$ and $n$ are fixed.
This raises the following question:

\begin{problem}
Let $D(k,n)$ be the smallest integer $d$ such that
the $k$-th mixtures of Gaussians on $\RR^n$
are algebraically identifiable from their moments
of order $\leq d$. Find good upper bounds on $D(k,n)$.
What are the best bounds that can be derived using algebraic geometry methods?
\end{problem}

For $n \geq 2$ it is difficult to compute the prime ideal of the
Gaussian moment variety $\mathcal{G}_{n,d}$ in $\PP^N$.  One approach
is to work on the affine open set $\mathbb{A}^{\!N} = \{m_{00\cdots0} = 1\}$.
On that affine space, $\mathcal{G}_{n,d}$ is a complete intersection
defined by the vanishing of all cumulants $k_{i_1 i_2 \cdots i_n}$
whose order $i_1+i_2+ \cdots +i_n$ is between $3$ and $d$;
see \cite[Remark 6]{CB}. Each such cumulant is a polynomial in the moments. 
Explicit formulas are obtained from the identity $K  = {\rm log}(M)$ of generating functions;
see \cite[eqn (8)]{CB}. The ideal of $\mathcal{G}_{n,d}$ is then obtained
from the ideal of cumulants by saturating with respect to $m_{00\cdots 0}$.
One example is featured in \cite[eqn (7)]{CB}.

\smallskip

We next exhibit an alternative representation of
$\mathcal{G}_{n,d} \cap \mathbb{A}^{\!N}$
as a determinantal variety. This is derived from 
Willink's recursion in \cite{Wil}. It generalizes the matrix $G_d$ in
Proposition~\ref{prop:surface}.
We define the {\em Willink matrix} $W_{n,d}$ as follows.
Its rows are indexed by vectors $u \in \mathbb{N}^n$
with $|u| \leq d-1$.
 The matrix $W_{d,n}$ has $2n+1$ columns. The first entry in the row $u$
is the corresponding moment $m_u$. The next $n$ entries in the row $u$
are  $\,m_{u+e_1}, \, m_{u+e_2} , \,\ldots ,\,m_{u+e_n}$.
The last $n$ entries in the row $u$ are 
$\,u_1 m_{u-e_1},\,u_2 m_{u-e_2},\,\ldots,\, u_n m_{u-e_n}$.
Thus the Willink matrix $W_{n,d}$ has
format $\binom{n+d-1}{d-1} \times (2n+1)$ and each entry is
a scalar multiple of one of the moments.
For $n = 1$, the $d \times 3$-matrix $W_{1,d}$ equals
the transpose of the matrix $G_d$ after permuting rows.

\begin{proposition}
The affine Gaussian moment variety $\mathcal{G}_{n,d} \cap \mathbb{A}^{\!N}$
is defined by the vanishing of the $(n+2) \times (n+2)$-minors of the
Willink matrix $W_{n,d}$.
\end{proposition}

\begin{proof}
Suppose that the matrix $W_{n,d}$ is filled with the moments of
a Gaussian distribution on $\RR^n$, and consider the 
$n$ linearly independent vectors
\begin{equation}
\label{eq:basis}
  \qquad \bigl(\,\mu_i, \,0,\ldots,0,-1,0,\ldots,0, \, \sigma_{1i}, \sigma_{2i},\ldots, \sigma_{ni}\,\bigr)^T
\qquad \qquad \hbox{for}\, \,i=1,2,\ldots,n. 
\end{equation}
Here the entry $-1$ appears in the $(i+1)$st coordinate. By \cite[eqn (13)]{Wil},
these $n$ vectors are in the kernel of $W_{n,d}$. Hence the rank of $W_{n,d}$ 
is $\leq n+1$, and the $(n+2)$-minors are zero.

Conversely, let $m$ be an arbitrary point in $\mathbb{A}^{\! N}$
for which the matrix $W_{n,d}$ has rank $\leq n+1$.
The square submatrix indexed by the rows $1,2,\ldots,n+1$
and the columns $1,n+2,\ldots,2n+1$ has determinant equal to
$m_{00\cdots 0}^{n+1} = 1$. Hence the rank of $W_{n,d}$ is exactly $n+1$.
The kernel of the submatrix given  by the first
$n+1$ rows is an $n$-dimensional space 
for which we can pick a basis of the form (\ref{eq:basis}).
The entries can be interpreted as the
mean and the covariance matrix of a Gaussian distribution.
The rank hypothesis on $W_{n,d}$ now ensures that the $n$
vectors in (\ref{eq:basis}) are in the kernel of the full matrix $W_{n,d}$.
This means that the higher moments satisfy the recurrences in
\cite[eqn (13)]{Wil}, and hence the chosen point $m$ lies in 
$\mathcal{G}_{n,d}$.
\end{proof}

\begin{example} \rm
Consider the moments of order at most four for a bivariate Gaussian.
The variety $\mathcal{G}_{2,4}$ has dimension $5$
and degree $102$ in $\PP^{14}$.
Its Willink matrix has format $10 \times 5$:
$$ W_{2,4} \,\, = \,\, \begin{pmatrix}
\, m_{00} & m_{01} & \, m_{10} &    0 &     0 \\
\, m_{01} & m_{02} & \, m_{11} &    0 &   m_{00} \\
\, m_{10} & m_{11} & \, m_{20} &  m_{00} &     0 \\
\, m_{02} & m_{03} & \, m_{12} &  0 & 2 m_{01} \\
\, m_{11} & m_{12} & \, m_{21} &  m_{01} &   m_{10} \\
\, m_{20} & m_{21} & \, m_{30} & 2 m_{10} &     0 \\
\, m_{03} & m_{04} & \, m_{13} &    0 & 3 m_{02} \\
\, m_{12} & m_{13} & \, m_{22} &  m_{02} & 2 m_{11} \\
\, m_{21} & m_{22} & \, m_{31} & 2 m_{11} &   m_{20} \\
\, m_{30} & m_{31} & \, m_{40} & 3 m_{20} &     0
\end{pmatrix}
$$
The ideal of $4 \times 4$-minors of $W_{2,4}$ is minimally generated
by $657$ quartics. Saturation with respect to the coordinate $m_{00}$ yields the
prime ideal of $\mathcal{G}_{2,4}$, as described in \cite[Proposition~7]{CB}.
\end{example}

One would expect that it is even more difficult to describe the prime ideals of 
the secant varieties ${\rm Sec}_k( \mathcal{G}_{n,d})$ for $n \geq 2$, $k>1$. Actually, it is
already an open problem to find these ideals when $n=1$, $k=2$ and $d \geq 8$. We found in \cite[Theorem 1]{CB} that
 ${\rm Sec}_2( \mathcal{G}_{1,6})$ is a hypersurface of degree $39$ in $\PP^6$. Its defining polynomial is
the  sum of  $   31154  $ monomials.

\begin{example} \label{s2G17} \rm Let $n=1,\, k=2$ and $d=7$. The following results
 were obtained using methods from numerical algebraic geometry.
The $5$-dimensional variety ${\rm Sec}_2( \mathcal{G}_{1,7})$  has degree  $105$ in $\PP^7$. The eight coordinate projections,
defined algebraically by eliminating each one of $m_{07}, m_{16}, \ldots, m_{70}$ from the ideal of ${\rm Sec}_2( \mathcal{G}_{1,7})$,
are hypersurfaces in $\PP^6$. Their degrees are $\,85,99,104,95,78,66,48$ and $39 $ respectively.
 This suggests that there are no low degree generators in the ideal of ${\rm Sec}_2( \mathcal{G}_{1,7})$.
In fact, a state-of-the-art Gr\"obner basis computation by Jean-Charles Faug\`ere shows that the smallest degree
of such a  minimal generator is $25$.
\end{example}

With ideal generators  out of reach, we first ask for the degrees of our secant varieties.

\begin{conjecture}  \label{conj:bold}
For fixed $k$ and $n$, the function  $d \mapsto {\rm deg}\, {\rm Sec}_k(\mathcal{G}_{n,d})$ is a polynomial in $d$,
starting from the smallest value of $d$ where the secant variety does not fill the ambient space.
\end{conjecture}

The  numerical {\tt Macaulay2} \cite {MAC2}  package \texttt{NumericalImplicitization.m2},
developed by Chen and Kileel \cite{JJ}, was very useful for us.
It was able to compute the desired degrees in some interesting cases. These data points led us to
Conjecture~\ref{conj:bold} and  to the following result.

\begin{proposition} \label{conj:7432}
Suppose that Conjecture~\ref{conj:bold} holds for $k=2$ and $ n=1$.
Then, for all $d \geq 6$, the degree of the $d$th moment variety for mixtures of two univariate Gaussians equals
\begin{equation} \label{eq:degs2G}
\deg \,{\rm Sec}_2( \mathcal{G}_{1,d})\,\, =\,\, \frac{(d+7)(d-4)(d-3)(d-2)}{8}.
\end{equation}
\end{proposition}

\begin{proof}
Let $X_d$ be a general variety defined by
a Hilbert-Burch matrix $B_d$ as in (\ref{eq:hilbertburch}). Here `general' means that the entries in $B_d$ 
are generic linear forms in $x,y,z$. 
Using the double point formula in intersection theory \cite[Sec.~9.3]{Ful} for a general projection $X_d\to \PP^4$, we~compute
\begin{equation} \label{eq:degs2X}
\deg \, {\rm Sec}_2(X_d) \,\, =\,\, \frac{(d-4)(d-3)(d^2+5d-2)}{8}. 
\end{equation}  
Since $\mathcal{G}_{1,d}$ is singular, the degrees of its secant varieties are
 lower than (\ref{eq:degs2X}), with a correction term accounting for the singular line in Lemma \ref{lem:Gsing}. 
 The assumption that Conjecture~\ref{conj:bold} holds in our case  implies
 that   $d \mapsto {\rm Sec}_2( \mathcal{G}_{1,d})$ is a polynomial function of degree at most~$4$.
 Our numerical computation shows that the degrees of ${\rm Sec}_2( \mathcal{G}_{1,d})$ for $d=6,\ldots,10$   are 
 $ 39$, $105$, $225$, $420$ and $714$. These are enough to interpolate, and we obtain the polynomial in (\ref{eq:degs2G}).
 \end{proof}
 
 \begin{remark} \rm
The zeroes of \eqref{eq:degs2G} at $d=2,3,4$ were not part of the interpolation but they are not unexpected. Also, substituting $d=5$ into \eqref{eq:degs2G} recovers the famous degree $9$ that was found by Pearson in 1894 for
identifying mixtures of two univariate Gaussians \cite[Sec.~3]{CB}.
Using \texttt{NumericalImplicitization.m2}, we verified the correctness of \eqref{eq:degs2G} up to $d=11$. 
\end{remark} 

Following this train of thought, and using the Le Barz classification formulas in \cite{Barz}, we compute an analogous formula 
to (\ref{eq:degs2X}) for trisecants ($k=3$) of a general surface $X_d$:
$$\deg({\rm Sec}_3(X_d))\,\,=\,\, \frac{(d-6)(d^5+3d^4-57d^3-43d^2+752d-512)}{48}. $$
 Conjecture \ref{conj:bold} now suggests that
$\,d \mapsto \deg \,{\rm Sec}_3( \mathcal{G}_{1,d})\,$ is a polynomial function of degree~$6$.
Unfortunately, we do not yet have numerical evidence for this. 
For instance, we do not even know the degree of $\,{\rm Sec}_3(\mathcal{G}_{1,9})$.
The formula yields the upper bound $\,\deg({\rm Sec}_3(X_9))=2497$.

\smallskip

We close with two more cases with $n \geq 2$ for which we were able to compute the degrees.

\begin{example} \rm
Let $n = 2 $ and $d=4$.
The $5$-dimensional  moment variety $\mathcal{G}_{2,4}$
has degree $102 $ in $\mathbb{P}^{14}$.
It is not defective. Its secant variety
${\rm Sec}_2(\mathcal{G}_{2,4})$ has dimension $11$
and degree $538$.
\end{example}

\begin{example} \rm 
We return to Example~\ref{s2G33}, so $n=d=3$.
The Gaussian moment variety $\mathcal{G}_{3,3}$ has dimension $9$ and degree $130$ in $\mathbb{P}^{19}$.
The number $130$ was reported in \cite[Sec. 2]{CB}.  This variety is $2$-defective.
Its secant variety ${\rm Sec}_2(\mathcal{G}_{3,3})$ has dimension $17$ and degree $79$.
We do not know its ideal generators.
As in Example \ref{s2G17}, we studied the degrees of its coordinate projections. The $20$ coordinates on $\mathbb{P}^{19}$
come in seven symmetry classes. Representatives are $m_{000},m_{100},m_{200},m_{110},m_{300},m_{201},m_{111}$.
By omitting these coordinates, one at a time,
 we obtain hypersurfaces in $\mathbb{P}^{18}$ whose degrees are  $58,63,34,42,25,34$ and $40$ respectively.
\end{example}

\medskip
\bigskip

\noindent
{\bf Acknowledgements.}
We thank the Max-Planck-Institute for Mathematics in the Sciences,
Leipzig, for its hospitality during our work on this project.
Carlos Am\'endola and Bernd Sturmfels were supported by 
the Einstein Foundation Berlin.
Bernd Sturmfels also acknowledges funding from
the US National Science Foundation (DMS-1419018).
Kristian Ranestad acknowledges funding from
the Research Council of Norway (RNC grant 239015).

\bigskip

\begin{small}

\end{small}

\bigskip \medskip

\noindent
\footnotesize {\bf Authors' addresses:}

\smallskip

\noindent Carlos Am\'endola,
Technical University Berlin, Germany,
{\tt amendola@math.tu-berlin.de}

\noindent Kristin Ranestad,
University of Oslo, Norway,
{\tt ranestad@math.uio.no}

\noindent Bernd Sturmfels, 
University of California, Berkeley, USA,
{\tt bernd@berkeley.edu} \\ \phantom{Bernd Sturmfels,}
and \ MPI-MiS Leipzig, Germany, {\tt bernd@mis.mpg.de}

\end{document}